\documentclass[a4paper,11pt]{article}
\usepackage{amsmath}
\usepackage[utf8]{inputenc}
\usepackage[T1]{fontenc}

\usepackage{amsthm}
\usepackage{tikz}
\usepackage{mathrsfs,amssymb,amsfonts} 
\usepackage{enumitem}
\usepackage{fullpage}
\usepackage{hyperref, enumerate}
\usepackage[babel]{microtype}
\usepackage[english]{babel}
\usepackage[capitalise]{cleveref}

\usepackage{thmtools}
\usepackage{mathtools, comment}
\usepackage{amssymb}
\usepackage[nomath]{lmodern}
\usepackage{graphicx}
\usepackage{pgf,tikz,tkz-graph,subcaption}
\usetikzlibrary{arrows,shapes}
\usetikzlibrary{decorations.pathreplacing}
\usepackage{tkz-berge}
\usepackage{enumitem}
\usepackage[normalem]{ulem}
\usepackage{hyperref}
\hypersetup{colorlinks = true, linkcolor = blue, citecolor = blue, urlcolor = blue}

\newcommand*{\abs}[1]{\lvert #1\rvert}

\allowdisplaybreaks

\usepackage[margin=1in]{geometry}
\newtheorem{innercustomthm}{Theorem}
\newenvironment{customthm}[1]
 {\renewcommand\theinnercustomthm{#1}\innercustomthm}
 {\endinnercustomthm}

\newtheorem{defi}{Definition}

\newtheorem{thm}[defi]{Theorem}

\newtheorem{claim}[defi]{Claim}
\newcommand*{\myproofname}{Proof}
\newenvironment{claimproof}[1][\myproofname]{\begin{proof}[#1]}{\end{proof}}

\DeclareMathOperator{\irr}{irr}
\DeclareMathOperator{\Var}{Var}

\author{Stijn Cambie
 \thanks{Department of Computer Science, KU Leuven Campus Kulak-Kortrijk, 8500 Kortrijk, Belgium. Supported by a postdoctoral fellowship by the Research Foundation Flanders (FWO) with grant number 1225224N. Email: \protect\href{mailto:stijn.cambie@hotmail.com}{\protect\nolinkurl{stijn.cambie@hotmail.com}}} \and Jionghua Chang  \thanks{Center for Combinatorics and LPMC, Nankai University, Tianjin 300071, P.R. China. Email:\protect\href{jchang1216914495@gmail.com}{\protect\nolinkurl{jchang1216914495@gmail.com}}}}

\title{Maximum ratio of (graph) irregularities}

\begin{document}
\parindent=0cm
\maketitle

\begin{abstract}
    We estimate the maximum ratio between the $\sigma_t$- and $\sigma$-irregularity for graphs and trees of order $n$, which are respectively bounded by $\Theta(n^{5/2})$ and $n-2$. This answers a question and a conjecture by Filipovski et al. in an elegant way.
    For trees, we obtain that the (Albertson) irregularity measure $\irr$ is an upper bound for the graph variance (normalised with the order).
\end{abstract}

\section{Introduction}
Over the years, one has defined multiple measures of irregularity for graphs. See e.g.~\cite{ADG19} for a discussion of these.
Among the most frequently used ones, there is the irregularity $\irr$ (or $\irr_A$) defined by Albertson~\cite{Albertson97} and the degree variance $\Var$, as defined by Snijders~\cite{snijders81} and Bell~\cite{Bell92}.
For a graph $G=(V, E)$ with order $n$ and size $m$, $\irr(G)=\sum_{uv\in E(G)}\abs{d(u)-d(v)},$ where $d(u)$ is the degree of a vertex $u \in V(G),$ and
$\Var(G)= \frac{\sum_{u\in V(G)}(d(u)-\overline{d})^{2}}{n},$ where $\overline d=\frac{2m}{n}$ is the average degree of the graph $G$.
The latter is related to the indices $IRV(G)$ (defined by R\'{e}ti~\cite{Reti19}) or $\sigma_{t}(G)=\sum_{\{u,v\}\subseteq V(G)}(d(u)-d(v))^2$ (defined in~\cite{DS23}), which both equal $n^2\Var(G).$

As a first result, we note a relation between two of these most important irregularity measures, $\irr$ and $\Var$, for trees.

\begin{thm}\label{thm:1}
    For every tree $T$ of order $n \ge 3$,
    $$\irr(T) > n \Var(T).$$
\end{thm}

Another irregularity measure $\sigma$ (\cite{FGKZP15,GYCC18}), is defined as
$\sigma(G)=\sum_{uv\in E(G)}(d(u)-d(v))^2.$ 
It is interesting to remark that the graphs maximising the irregularity measures $\sigma$ and $\sigma_t$ have only two different degrees, see e.g.~\cite{ADG18, FDKS24}.

Being the sum over a subset of the summands of $\sigma_t,$ it is clear that $\sigma(G)\le \sigma_t(G)$ for every graph $G$. Equality is easily characterised (\cite[Prop.~23]{FDKS24}), with $S_n$ (being the sole tree for which all non-leaves are adjacent to all leaves) the only tree attaining equality.
In the other direction, as a non-trivial corollary of a stronger version of~\cref{thm:1}, we deduce the following sharp inequality between $\sigma_t$ and $\sigma$ for trees.

\begin{thm}\label{thm:2}
    For a tree $T$ of order $n$, $\sigma_t(T)\le (n-2) \sigma(T).$
    Equality is only attained by $T=P_n.$
\end{thm}

This proves~\cite[Conj.~25]{FDKS24}.
The fact that equality is only attained by $P_n$ may suggest the inequality is easy to prove. Nevertheless, it turned out to be more challenging since the inequality is asymptotically tight for many graphs (e.g. consider the greedy trees where the degree of neighbours differ by exactly $1$).

For general graphs, answering a question posed by~\cite{FDKS24} (preceding~\cite[Conj.~25]{FDKS24}), we determine the order of the ratio $\frac{\sigma_t(G)}{\sigma(G)}$ for every connected (irregular) graph $G$.

\begin{thm}\label{thm:3}
    For every (connected) graph $G$, $\sigma_t(G) \le f(n) \sigma(G)$ for some function $f(n)= \Theta(n^{5/2})$. This is sharp, and $f(n) \le \sqrt{1.5} n^{5/2}.$
\end{thm}

Perhaps surprisingly, this function is larger than $n^2$.
Note that a cubic upper bound is trivial, and a quadratic lower bound can be obtained by connecting the ends of a path $P_{n/2}$ with the low-degree vertices of $K_{n/2}^-$ ($K_{n/2}$ minus one edge). The correct order is in between.
The proof relies on a well-chosen subset of edges (part of a path with an extremal property) which we use to estimate $\sigma(G)$ and a construction attaining the bound, composed by connecting multiple regular parts by a path.

\section{Main proofs}



We start with proving the following stronger version of~\cref{thm:1}.
Note for this that $\sum_{u\in V(G)}(d(u)-2)^{2}>\sum_{u\in V(G)}(d(u)-\overline{d})^{2},$ the latter because the average (here $\overline d$) minimises the mean square error (here sum of differences to the degrees $d(v)$ squared), and $\Delta(T) \ge 2.$

\begin{thm}\label{thm:conj25_elegant}
    For every tree $T=(V, E)$ with order at least $3$,
    $$\sum_{e=uv \in E} \abs{d(u)-d(v)} \ge \sum_{v \in V} ( d(v)-2)^2 +2(\Delta(T)-2)  .$$
\end{thm}

\begin{proof}
    Let $T$ be a tree which is a counterexample.
    By~\cite[Thm.~2.1]{DGLC23} (the proof works analogously for minimising $\irr(T)$ instead of $\sigma(T)$), we can assume it is the greedy tree with a certain degree sequence (the right hand side is fixed by the degree sequence, while the left hand side is not).
    Greedy trees are constructed from a given degree sequence by assigning the highest degree $\Delta$ to the root, the second-, third-, $(\Delta+1)^{th}$ highest degrees to the root's neighbours, and so on.

    We consider this greedy tree as a rooted tree with the root $r$ having maximum degree, drawn with the root on top and children of every vertex drawn below.

    For every vertex $x$, let $T_x$ be the tree rooted in $x$ with its descendants (the vertex $x$ with everything connected to it below).
    
    We prove by induction on the height of $T_x$ that for every $x \not =r$,
    $$ g(T_x)=\sum_{e=uv \in E(T_x) } \abs{d(u)-d(v)} - \sum_{v \in V(T_x)} ( d(v)-2)^2 = d(x)-2.$$

    The base case is true, since for a leaf $x$ it states $0-1=1-2.$
    
    If it is true for the children of vertex $x$,
    with children $Y$, then
    
\begin{align*}
    g(T_x)&= - (d(x)-2)^2+ \sum_{y \in Y} (d(x)-d(y))  + \sum_{y \in Y} g(T_y)\\&= - (d(x)-2)^2 +  \sum_{y \in Y} (d(x)-2) \\&= d(x)-2. 
\end{align*}

Finally, for the root $r$, the same computation leads to 
$$\hspace{2cm} g(T_r)= -(d(r)-2)^2 + d(r)(d(r)-2)= 2(d(r)-2)=2(\Delta(T)-2). \hspace{2cm} \qedhere$$
\end{proof}

From~\cref{thm:conj25_elegant}, we also deduce~\cref{thm:2} (and thus~\cite[Conj.~25]{FDKS24}). For the reader’s convenience, we restate it (the equality for $\Delta=2$ is trivial to verify).

\begin{customthm} {\bf \ref{thm:2}}
    For a tree $T$ of order $n$ with $\Delta(T)\ge3$, $\sigma_t(T)<(n-2) \sigma(T).$
\end{customthm}

\begin{proof}
    If $\Delta=3,$ let the tree have $x+2$ leaves, $x$ vertices of degree $3$ and $y$ vertices of degree $2.$

    Since a lower bound for $\sigma$ is obtained by the greedy tree, we deduce that 
    $$\sigma(T) \ge \begin{cases}
        4(x+2)-2y &\mbox{ if } y \le x+2\\
        2(x+2) &\mbox{ if } y \ge x+2.\\
    \end{cases}$$
    In particular $\sigma(T) \ge 2(x+2)$, and we conclude since $$(n-2)\sigma(T) \ge 2(2x+y)(x+2) \ge 4x(x+2)+2y(x+1)= \sigma_t(T).$$
    Furthermore, one can check that no equality can occur. 
    If $y=0,$ then $\sigma(T) \ge 4(x+2)>2(x+2)$ and thus the first inequality is strict. If $y>0$ the second inequality is strict (the difference is $2y$).

    If $\Delta \ge 4$, we can conclude from~\cref{thm:conj25_elegant}.
    First we notice the following claim.
    \begin{claim}\label{clm:uppbnd_sum(d-2)^2}
        A tree $T$ with maximum degree $\Delta$ and order $n$ satisfies
        \begin{equation}\label{eq:1}
            \sum_{v \in V} ( d(v)-2)^2 \le (n-2)(\Delta-2)+2
        \end{equation}
    \end{claim}
    \begin{claimproof}
        By Karamata's inequality, the left hand side of~\eqref{eq:1} is upper bounded by \begin{align*} \hspace{2cm}2 \cdot (1-2)^2+\frac{n-2}{\Delta-1} \left( (\Delta-2)^2+(\Delta-2)(1-2)^2\right)= 2+(n-2)(\Delta-2). \hspace{2cm} \qedhere \end{align*}
    \end{claimproof}

    If the tree has no $2$ neighbouring vertices for which the degree differ by more than one, then~\cref{clm:uppbnd_sum(d-2)^2} is also true without the $+2$.
    Karamata's estimate assumes all degrees are in $\{\Delta,1 \}$, while $ \Delta-1$ and $ 2$ are also degrees.
    Now $(\Delta-2)^2 + (1-2)^2 \ge (\Delta-1 -2)^2+(2-2)^2+2$, thus the difference with the estimate of the sum of squares in the left hand side of~\eqref{eq:1} is at least $2$.
    Together with~\cref{thm:conj25_elegant}, this implies that
     $$ \sigma(T) \ge \irr(T) \ge \frac{n}{n-2} \sum_{v \in V} ( d(v)-2)^2 > \frac{n}{n-2} \sum_{v \in V} ( d(v)-\overline d)^2 ,$$
     and thus $(n-2)\sigma(T)> \sigma_t(T).$

     If the tree has an edge $uv$ for which $\abs{d(u)-d(v)} \ge 2$, and we consider $\sum_{e=uv \in E} (d(u)-d(v))^2$ instead of $\sum_{e=uv \in E} \abs{d(u)-d(v)},$ then the analogue of~\cref{thm:conj25_elegant} holds with $2$ added to the right hand side.
     Since $(n-2)(\Delta-2)+2 < (n-2)(\Delta-1)$, we conclude that $\sigma(T) > \frac{n}{n-2} \sum_{v \in V} ( d(v)-\overline d)^2=\frac{\sigma_t(T)}{n-2}$, as before.
\end{proof}

Finally, we prove~\cref{thm:3}, repeated here for convenience of the reader.

\begin{customthm} {\bf \ref{thm:3}}
    For every (connected) graph $G$, $\sigma_t(G) \le f(n) \sigma(G)$ for some function $f(n)= \Theta(n^{5/2})$. This is sharp, and $f(n) \le \sqrt{1.5} n^{5/2}.$
\end{customthm}

\begin{proof}[Proof of upper bound $f(n)= O(n^{5/2})$]
    Let $P$ be a path from a vertex of minimum degree to one of maximum degree, such that the (strictly) increasing subsequence of its (ordered) degree sequence one obtains by greedily adding larger degrees has minimum length $r+1$ among all such paths (and choose one of minimum length).


    \begin{claim}
        $r \le \sqrt{6n}$.
    \end{claim}

    \begin{claimproof}
        Let the degrees of the particular minimum (strictly) increasing subsequence be $1  \le d_0 < d_1< \ldots< d_{r}$.
        Note that $\sum d_i \ge \sum_{i=0}^{r} (i+1) =\binom{r+2}{2}.$
        Let $v_0, v_1, \ldots, v_r$ be the corresponding vertices with these degrees (one representative) of the path.
        Note that no vertex has more than $3$ neighbours among $\{v_0, v_1, \ldots, v_r\},$ as otherwise one could take a path with shorter greedy increasing subsequence.
        This implies that $\frac{r^2}2<\binom{r+2}{2} \le 3n$, from which the conclusion follows.
    \end{claimproof}

    By definition, $\sigma(G) \ge \sum_{uv \in E(P)} (d(u)-d(v))^2 \ge \sum_{i=1}^r ((d_i - d_{i-1})^2 \ge r \left(\frac{ \Delta- \delta}{r}\right)^2 =\frac{ (\Delta-\delta)^2}{r}.$
    Here AM-GM is used.

    Since $\sigma_t(G) \le \binom{n}{2}(\Delta-\delta)^2$ by the trivial upper bound, we conclude that 
    $\frac{\sigma_t(G)}{\sigma(G)}\le r\binom{n}{2} \le \sqrt{6n}\frac{n^2}2=O(n^{5/2}).$
\end{proof}
Next, we prove that the bound $f(n)= \Theta(n^{5/2})$ is sharp.

\begin{proof}[Proof of $f(n)= \Omega(n^{5/2})$]
We make use of the following fact.
\begin{claim}
    Let $k\ge 8$ be an even number. For every $4\le r \le k-4$, there exists a (connected) graph $G_r$ of order $k$ such that all vertices of $G_r$ have degree $r$, except from one vertex $v_r$ which has degree $r-2.$
\end{claim}
\begin{claimproof}
    Notice that $K_k$ can be decomposed in $\frac k2-1$ two-factors (spanning $2$-regular subgraphs) and a perfect matching, by ordering the vertices as $1\le v \le k$ and for every $1 \le i \le \frac k2$ connecting the vertices which differ by $i$ in the ordering (where the difference is taken modulo $k$).
    One can also take the cycle on $1 \le v \le k-1$ and connect the ones with difference $2$ modulo $k-1$.
    The union of this latter cycle with some of the previous two-factors for $i \not \in \{2,3\}$ and possibly the perfect matching, gives the result.
\end{claimproof}

For every integer $s\ge 0$, one can also consider a graph $G_3$ and $G_{k-3}$ which are of order $k+1+2s$ whose vertices all have degree $3$ respectively $k-3$ but one vertex (called $v_3$ respectively $v_{k-3}$) which has degree one less. The proof of existence is analogous to the previous claim.

Now take $\cup_{i=3}^{k-3} G_i $ and for every $3 \le i \le k-4$, connect the vertices $v_i$ and $v_{i+1}.$
Thus all initial vertices from $G_i$ have now degree $i$.
The resulting graph $G$ has order $k(k-5)+2+4s.$
All edges in the resulting graph $G$ have end vertices with the same degree, except for the edges $v_iv_{i+1}.$
This implies that $\sigma(G)=k-6.$

On the other hand $\sigma_t(G)$ is of the order $n^2k^2 \sim k^6$ (for $s$ being bounded by $k$, and thus $n \sim k^2$), resulting in $\frac{\sigma_t(G)}{\sigma(G)}$ being of the order $\frac{k^6}{k} \sim n^{5/2}.$
For this, it is sufficient to note that there are $\binom n2 = \Theta(n^2)$ terms, and $(d(u)-d(v))^2$ is typically of the order $k^2$.
For odd order, one can modify $G_4$ with one additional vertex, and conclude that the construction works for all large $n.$
\end{proof}

\section*{Acknowledgement}

This research was conducted during S.C.'s visit to Nankai University, hosted by Yongtang Shi and Jiangdong Ai. S.C. gratefully acknowledges their hospitality. We are grateful to the referees for suggestions improving the presentation of this work.


\begin{thebibliography}{1}

\bibitem{Albertson97} M. O. Albertson, The irregularity of a graph, \emph{Ars Comb.} 46 (1997) 219--225.

\bibitem{ADG18} H. Abdo, D. Dimitrov, I. Gutman, Graphs with maximal $\sigma$ irregularity, \emph{Discrete Appl. Math.} 250 (2018) 57--64.

\bibitem{ADG19} H. Abdo, D. Dimitrov, I. Gutman, Graph irregularity and its measures, \emph{Appl. Math. Comput.} 357 (2019) 317--324.

\bibitem{Bell92}
F. K. Bell,  A note on the irregularity of graphs, Linear Algebra Appl. 161 (1992), 45--54.
 

\bibitem{DGLC23}
D.~Dimitrov, W.~Gao, W.~Lin, and J.~Chen.
\newblock Extremal trees with fixed degree sequence for {{\(\sigma\)}}-irregularity.
\newblock {\em DML, Discrete Math. Lett.}, 12:166--172, 2023.

\bibitem{DS23} D. Dimitrov, D. Stevanović, On the $\sigma_{t}$-irregularity and the inverse irregularity problem,
\emph{Appl. Math. Comput.}  441 (2023) 127709.

\bibitem{FDKS24}
S.~Filipovski, D.~Dimitrov, M.~Knor, and R.~{\v{S}}krekovski.
\newblock Some results on $\sigma_{t}$-irregularity.
\newblock Preprint, {arXiv}:2411.04881 [math.{CO}] (2024), 2024. Accepted at Ars Math. Contemp., doi.org/10.26493/1855-3974.3268.60f

\bibitem{FGKZP15}
B.~Furtula, I.~Gutman, {\v{Z}}.~Kovijani{\'c}~Vuki{\'c}evi{\'c}, G.~Lekishvili, and G.~Popivoda.
\newblock On an old/new degree-based topological index.
\newblock {\em Bull., Cl. Sci. Math. Nat., Sci. Math.}, 40:19--31, 2015.

\bibitem{GYCC18} I. Gutman, M. Togan, A. Yurttas, A. S. Cevik, I. N. Cangul, Inverse problem for sigma index, \emph{MATCH Commun. Math. Comput. Chem.} 79 (2018) 491--508.

\bibitem{Reti19} T. R\' {e}ti, On some properties of graph irregularity indices with a particular regard to the $\sigma$-index, \emph{Appl. Math. Comput.} 344--345 (2019) 107--115.

\bibitem{snijders81} T. A. B. Snijders, The Degree Variance, An Index of Graph Heterogeneity, \emph{Social Networks} 3 (1981) 163--174.

\end{thebibliography}
\end{document}